\documentclass{amsart}
\usepackage{graphicx}
\usepackage{indentfirst}
\usepackage{hyperref}

\usepackage{latexsym}
\usepackage{amsmath}
\usepackage{amsthm}
\usepackage{amssymb}
\usepackage{pdfpages}
\usepackage[labelformat=empty]{caption}
\usepackage{CJKutf8}

\newtheorem{theorem}{Theorem}[section]
\newtheorem{lemma}[theorem]{Lemma}

\theoremstyle{definition}

\newtheorem{prop}[theorem]{Proposition}
\newtheorem{remark}[theorem]{Remark}

\newcommand{\vip}{\vskip.2cm}
\newcommand{\R}{\mathbb{R}}
\newcommand{\T}{\widetilde}

\newcommand{\E}{\mathbb{E}}

\newcommand{\mL}{\mathcal{L}}

\newcommand{\indiq}{{\bf 1}}

\newcommand{\clim}{\lim_{n\to\infty}}
\newcommand{\cint}{\int_0^\infty}

 \makeatletter
\@namedef{subjclassname@2020}{ \textup{2020} Mathematics Subject Classification}
\makeatother

\hypersetup{pdfmenubar=true}

\begin{document}
\begin{CJK}{UTF8}{gbsn}

\title[Convergence for supercritical heavy tailed Hawkes processes]{SCALING LIMIT FOR SUPERCRITICAL NEARLY UNSTABLE  HAWKES PROCESSES WITH HEAVY TAIL}

\author[L. Xu]{Liping Xu}
\author[A.Zhang]{An Zhang}

\address{School of mathematical sciences, Beihang University, PR China}
\email{xuliping.p6@gmail.com}
\address{School of mathematical sciences, Beihang University, PR China}
\email{anzhang@buaa.edu.cn}



\begin{abstract} 
In this paper, we establish the asymptotic behavior of {\it supercritical} nearly unstable Hawkes processes with a power law kernel.  We find that, the Hawkes process in our context admits a similar equation to that in \cite{MR3563196} for {\it subcritical} case.  In particular, the rescaled Hawkes process $(Z^n_{nt}/n^{2\alpha})_{t\in[0,1]}$ converges in law to a kind of  integrated fractional Cox–Ingersoll–Ross process with different coefficients from that in \cite{MR3563196}, as $n$ tends to infinity.
\end{abstract}

\maketitle

\section{Introduction}
Hawkes process $(Z_t)_{t\ge0}$ first described by Hawkes \cite{MR278410} in 1971, is a kind of self-exciting stochastic point process, which means that the occurrence of subsequent random event is influenced by the occurrence of  former events. It has intensity given by  
\[\lambda_t=\mu+\int_0^t\varphi(t-s)dZ_s,\]
where $\mu>0$ and the regression kernel $\varphi:\R^+\to\R^+$ quantifies the influence of past  events of the process on the occurrence of future events.  
\vip

The explosion in popularity of Hawkes processes in the last years is attributed to its wide applications to model both natural and social phenomena in various fields, such as seismology \cite{MR278410,MR378093}, mostly finance \cite{MR3219705,MR3480107,MR3750729},  neuroscience \cite{MR3322314}, social networks interactions \cite{SIR2018,Social} and epidemiology \cite{MR4276450,MR4661701}, etc. From a mathematical perspective, there has been a bulk of  literature on Hawkes process, such as \cite{MR1950431} for an introduction, \cite{MR1411506} for stability results, \cite{MR3102513,MR3395721,MR4418230,MR3054533,MR3187500} for limit theorems, and \cite{MR2362700,MR3313748,MR4177368,MR3224291} for large deviations, etc. As is well-known, the great bulk of literature related to the Hawkes process and its applications  focuses on the {\it subcritical} case ($\|\varphi\|_1<1$), whereas the {\it supercritical} case ($\|\varphi\|_1>1$) is more germane to  epidemics, see \cite{MR4276450}. In this paper, we thus consider the  {\it supercritical nearly unstable} Hawkes process  with a {\it heavy-tailed} kernel of the form 
\[\varphi(t)\sim \frac{K}{t^{1+\alpha}}, \  \hbox{ for some constants } \alpha\in(1/2,1) \ \hbox{ and } K>0.\]
\vip
The  {\it nearly unstable} Hawkes process, referring to the case that the $L^1$ norm of its kernel tends to one, is first introduced by Jaisson and Rosenbaum in \cite{MR3313750} because it fits the data properly in finance according to the statistical stylized fact, even though it  almost violates the stability condition ($\|\varphi\|_1<1$). In \cite{MR3313750},  they proved that the suitably rescaled {\it subcritical nearly unstable} Hawkes process with {\it light-tailed} kernel converges in law to the integrated CIR process，that is well known in finance, and the corresponding rescaled intensity converges in law  to the CIR process. Recently, their work was extended to {\it supercritical} case, see \cite{LXZ2024} and to multivariate processes, see \cite{xu2023}. Whereas, concerning about the  {\it heavy tailed} kernel, which is in fact much more consistent in the financial data and related to the rough volatility in finance market, for more detailed explanation of this setting in finance, see e.g. \cite{MR3563196, MR3778355, MR4280453}. In \cite{MR3563196}, Jaisson and Rosenbaum  proved that the suitably rescaled Hawkes process converges weakly to the integral of a rough fractional diffusion. Subsequently, it was extended to the multivariate case as well, see \cite{MR3778355,MR4280453}. Recently, Horst et al. \cite{Horst2023} obtained a more refined convergence result from the rescaled intensities to a rough  fractional diffusion. Compared with \cite{MR3563196}, they obtained the weak convergence of the point process by establishing  directly the tightness of  the intensities, instead of their integrals, i.e. the Hawkes process itself. 
\vip
 In our paper, following from the road in \cite{MR3563196}, we consider the {\it supercritical} case and  prove that the weak convergence of the suitably rescaled Hawkes process $(Z^n_{nt}/n^{2\alpha})_{t\in[0,1]}$ exhibits a similar result to the subcritical case, that is, the weak convergence to an integral of a rough fractional diffusion as well, see section \ref{MT}. The main idea is that  we study the convergence of the Hawkes process itself as well and try to transfer the {\it supercritical} case to {\it subcritical} case as what we did in \cite{LXZ2024} making use of the Malthusian parameter.  However,  the light tailed kernel considered in  \cite{LXZ2024}, that is, the kernel $\varphi$ satisfies 
 \[\int_0^\infty s\varphi(s) ds=m<\infty,\] enables us to find easily the Malthusian parameter $b_n$ is similar to $n^{-1}$ as $n\to\infty$, which is important in the proof. Whereas, in our paper, the above condition is violated, the investigation of the limit behavior of $nb_n$ is one of the crucial problems in our paper.
  Finally, let's also mention the mean field limit of Hawkes process established in \cite{MR3499526,MR3449317, MR4127342}, and the functional limit of Hawkes process  investigated recently in \cite{Xu2024,CMR2025}.
\vip

\subsection{Setting}
We consider a  measurable  function $\varphi:[0,\infty)\to [0,\infty)$  and  a  Poisson measure $(\Pi(dt,dz))$,  on $[0,\infty)\times [0,\infty)$ with intensity $dtdz$. The sequence $\{a_n\}_{n\ge 1}$ indexed by $n$ which is positive.  We put $\varphi^n=a_n\varphi$ and  consider the following system indexed by $n$:  all $t\geq 0,$
\begin{align}\label{sssy}
Z_{t}^{n}:=\int^{t}_{0}\int^{\infty}_{0}\boldsymbol{1}_{\{z\le\lambda_{s}^{n}\}}\Pi(ds,dz), \hbox{ where }
\lambda_{t}^{n}:=\mu_n+\int_{0}^{t-}\varphi^n(t-s)dZ_{s}^{n},
\end{align}
where $\mu_n$ is a sequence of positive real numbers.
In this paper, $\int_{0}^{t}$ means $\int_{[0,t]}$, and $\int_{0}^{t-}$ means $\int_{[0,t)}$. The solution $(Z_{t}^n)_{t\ge 0}$ is a  counting processes. 
By \cite[Proposition 1]{MR3499526}, see also \cite{MR1411506,MR3449317}, the system $(1)$ has a unique $(\mathcal{F}_{t})_{t\ge 0}$-measurable c\`adl\`ag 
solution, where   
$$\mathcal{F}^n_{t}=\sigma(\Pi^{i}(A):A\in\mathcal{B}([0,t]\times [0,\infty)),$$ as soon as $\varphi^n$ is locally integrable.

\subsection{Assumption}
We consider a nonnegative measurable function $\varphi$ satisfying $ \int_0^\infty \varphi(t)dt=1$.
We assume that for some $\alpha\in(0,1/2)$ and some $K>0$, 
\renewcommand\theequation{{$H_1$}}
\begin{equation}\label{H1}
\lim_{x\to+\infty}\alpha x^{\alpha}\Big(1-F(x)\Big)=K,
\end{equation}
with $F(x)=\int_0^x \varphi(t)dt$. 

\vip

Recall that $\{a_n\}_{n\ge 1}$  is a sequence of positive numbers.  For some $\lambda> 0$, $\mu^*>0$ and $\delta:=K\frac{\Gamma(1-\alpha)}{\alpha}$ with $\Gamma$ the gamma function,   we  assume
\renewcommand\theequation{{$H_2$}}
\begin{equation}\label{H2}
\left\{\begin{array}{l}a_n> 1, \lim\limits_{n\to \infty} a_n=1.\\
 \lim\limits_{n\to\infty}n^\alpha(a_n-1)=\lambda\delta,\\
\lim\limits_{n\to\infty}n^{1-\alpha}\mu_n=\mu^*\delta.
\end{array}
\right.
\end{equation}
\renewcommand\theequation{\arabic{equation}}\addtocounter{equation}{-2}
The assumption $\lim\limits_{n\to\infty}n^{1-\alpha}\mu_n=\mu^*\delta$ implies that $\lim\limits_{n\to\infty}\mu_n=0$.

\subsection{Notation}\label{notation}
Let's first recall the convolution of two functions $f,g: [0,\infty)\mapsto \R$, which is defined by (If it exists) $(f*g)(t)=\int_0^tf(t-s)g(s)ds=\int_0^tf(s)g(t-s)ds$. $(f)^{* k}$ represents the $n$-th convolution product of function $f$ and $(f)^{*1}=f$.
We denote $L^1$ norm of a function $f$  by  $\|f\|_1= \int_0^\infty |f(s)|ds$. 
 
\vip
Since $\int_0^\infty \varphi(s)ds=1$,  we have
$\cint (\varphi^n)^{* k}(s)ds=(a_n)^k$.
For $ 1<a_n=\int_0^\infty \varphi^n(s)ds<+\infty,$ we can choose a  unique  positive sequence $\{b_n\}_{n\ge 1}$ such that
\begin{align}\label{bT}
    \int_0^\infty e^{-b_n s}\varphi^n(s)ds=\frac{1}{a_n}<1.
\end{align}
Then, we must have  
$\lim\limits_{n\to+\infty}b_n=0$ owing to  $\lim\limits_{n\to+\infty}a_n=1$.
We now introduce  the following functions on $\R^+$,
\begin{align}\label{phipsi}
    \T \varphi^n(t)=e^{-b_n t}\varphi^n(t), \ \T \Psi^n(t)=\sum_{k\ge 1}(\T\varphi^n)^{*k}(t).
\end{align}
Then, $\T \Psi^n(t)$  is well defined on $\R^+$ since $\| \T \varphi^n\|_1=1/a_n<1$.  
Moreover,
\begin{align}\label{norm}
\|\T \Psi^n\|_1=\frac{\| \T \varphi^n\|_1}{1-\| \T \varphi^n\|_1}=\frac{1}{a_n-1}.
\end{align}
We denote for $t\ge0$,
\begin{equation}\label{PSI2}
\Psi^n(t):=\sum_{k\ge 1}(\varphi^n)^{*k}(t),
\end{equation}
which is well defined as well following from the fact that
\begin{equation}\label{PSI}
\Psi^n(t)= e^{b_n t}   \T\Psi^n(t),
\end{equation}
due to $e^{-b_n t}( \varphi^n)^{*k}(t)=(e^{-b_n t} \varphi^n)^{*k}(t)=(\T \varphi^n)^{*k}(t).$

\vip
Let’s now introduce the following processes. For $t\in[0, 1]$, define the martingale
\begin{align}\label{mart}
M_t^n:=\int_0^t\cint \indiq_{\{z\le \lambda_s^n\}}\T \Pi(ds, dz),
\end{align}
where $\T\Pi(ds, dz)=\Pi(ds, dz)-dsdz$ is the compensated Poisson measure associated to the Poisson measure $\Pi$. We  then readily find that 
\[M_t^n=Z_{t}^{n}-\int_0^t \lambda^n_s ds.
\]
Moreover,  $[M^n,M^n]_t=Z_t^n$.
 \vip
For $t\in[0,1]$, consider the renormalized Hawkes process 
\[X_t^n=\frac{Z^n_{nt}}{n^{2\alpha}} \]
and its integrated intensity 
\[\Lambda_t^n=\frac{1}{n^{2\alpha}} \int^{nt}_0\lambda_{s}^{n}ds.\]
Denote the spatial renormalization of the martingale $M^n$  on $[0,1]$ by 
\[\overline{M}^n_t=\frac{M^n_{nt}}{n^{\alpha}}=n^\alpha(X_t^n-\Lambda_t^n).\]
It’s not hard to calculate that the quadratic variation 
$$[\overline{M}^n,\overline{M}^n]_t=X^n_t,\ t\in[0,1],$$ and the angle bracket of $\overline{M}^n$ is 
$$<\overline{M}^n, \overline{M}^n>_t=\Lambda^n_t$$ for $t\in[0,1]$.
Referring to \cite[Chapter 1, Section 4e]{MR1943877}  for definitions and properties of pure jump martingales and of their quadratic variations.

\vip
Let's now recall some basic properties of the Mittag-Leffler function which will be used. The reader may see \cite{MR2800586} for more interesting properties of the function.  For $\kappa, \beta>0$, the $(\kappa, \beta)$ Mittag-Leffler function on $\R$ is defined by 
$$
E_{\kappa,\beta}(x)=\sum_{n=0}^\infty \frac{x^n}{\Gamma(\kappa n+\beta)},\  x\in\R.
$$
For $\alpha\in(1/2,1)$ and $\lambda>0$ introduced in \eqref{H2}, define 
\begin{align}\label{MIF}
f^{\alpha,\lambda}(x):=\lambda x^{\alpha-1}E_{\alpha,\alpha}(\lambda x^{\alpha}),\  x>0, \quad F^{\alpha,\lambda}(x):=\int_0^xf^{\alpha,\lambda}(t)dt, \ x>0.
\end{align}
Given an integrable function $f$ defined on $\R^+$, we denote its Laplace transform by 
\[\mL f(z):=\int_0^\infty e^{-zt}f(t)dt, \  z\ge0.\]
It's not hard to compute the Laplace transform of $x^{\beta-1}E_{\kappa,\beta}(\lambda x^\kappa)$ which is equal to $\frac{z^{\kappa-\beta}}{-\lambda+z^\kappa}$. In particular, the Laplace transform of $f^{\alpha,\lambda}$ is of the form 
$$\mL f^{\alpha,\lambda}(z)=\int_0^\infty e^{-zt}f^{\alpha,\lambda}(t)dt=\frac{\lambda}{-\lambda+z^\alpha}, \quad z\ge0.
$$

\subsection{Main results}\label{MT}
In this part, we will give our main results.
\begin{theorem}\label{thm1}
 Under the assumption \eqref{H1} and \eqref{H2}, consider a  limit point $(X, M^*)$ of $(X^n, \overline{M}^n)$ defined in Proposition \ref{limpoi}. We have, $X$ satisfies 
\begin{align}\label{IPro}
X_t=\frac{\mu^*}{\lambda}\int_0^t f^{\alpha,\lambda}(t-s)s ds+\frac{1}{\lambda\delta}\int_0^t f^{\alpha,\lambda}(t-s)M^*_sds,
\end{align}
where $f^{\alpha,\lambda}$ is defined in \eqref{MIF}.
 \end{theorem}
\begin{remark}\label{RK}
Assume $\alpha>1/2$. By a similar proof of \cite[Theorem 3.1, 3.2]{MR3563196},  the limit point $(X_t)_{t\in[0,1]}$ satisfying \eqref{IPro} is H\"{o}lder continuous with exponent $(1\land2\alpha)-\epsilon$ and is differentiable with derivative denoted by $(Y_t)_{t\in[0,1]}$ which is continuous. This derivative process $(Y_t)_{t\in[0,1]}$ can be interpreted as a volatility process.
\end{remark}
According to the martingale representation theorem,  we have the following theorem for
 the process $(Y_t)_{t\in[0,1]}$.
\begin{theorem}\label{thm2}
Under the assumption \eqref{H1} and \eqref{H2}. Let $(Y_t)_{t\in[0,1]}$ be the derivative of the process $(X_t)_{t\in[0,1]}$ satisfying  \eqref{IPro}. Then, there exists a Brownian motion $B$, up to an enlargement of the probability space, such that  $(Y_t)_{t\in[0,1]}$ satisfies
\begin{align}\label{Vpro}
Y_t=\mu^*\delta F^{\alpha,\lambda}(t)+\frac{1}{\lambda\delta}\int_0^t f^{\alpha,\lambda}(t-s)\sqrt{Y_s}dB_s,
\end{align}
where $f^{\alpha,\lambda}$ is defined in \eqref{MIF}, $F^{\alpha,\lambda}$ is the cumulative function of $\frac{1}{\lambda\delta}f^{\alpha,\lambda}$ introduced in lemma \ref{KLP}. Moreover, the solution to the above equation is unique in law. 
\end{theorem}

\section{Preliminaries}
In this section, we will do some preparation. Let's begin with a primary result.
\begin{lemma}\label{BT}
Assume \eqref{H1} and \eqref{H2}, and recall the positive sequence of $\{b_n\}_{n\ge 1}$ introduced in \eqref{bT}, then for $\alpha\in(0,1)$ and $\lambda>0$ introduced in assumption \eqref{H2}, we have 
\[\lim_{n\to\infty}nb_n=(2\lambda)^{1/\alpha}.\]
\end{lemma}
\begin{proof}
We start with the  Laplace transform of $\varphi$ denoted by  $\mL\varphi$. It's not hard to  deduce from the integration by parts that 
\[
\mL\varphi(z)=z\int_0^{\infty}e^{-zt}F(t)dt=1-z\int_0^{\infty}e^{-zt}\big(1-F(t)\big)dt.
\]
Then recalling the assumption \eqref{H1} and the Karamata-Tauberian theorem \cite[Theorem 1.7.6]{MR1015093}, we have with $\delta=K\frac{\Gamma(1-\alpha)}{\alpha}$ introduced in \eqref{H2},
\begin{align}\label{LAP}
\mL\varphi(z)=1-\delta z^\alpha+o(z^\alpha).
\end{align}
According to \eqref{bT}, we have $\int_0^\infty e^{-b_nt}\varphi(t)dt=1/a_n^2$, that is, 
$\mL{\varphi}(b_n)=1/a_n^2$. Since $\lim\limits_{n\to\infty}b_n=0$, hence 
\[\lim_{n\to\infty} (1-\delta b_n^\alpha)=\lim_{n\to\infty}1/a_n^2,\] which is equivalent to 
\[\lim_{n\to\infty}b_n^\alpha=\lim_{n\to\infty}(1-1/a_n^2)\delta^{-1}.\] 
Noting that  $\lim\limits_{n\to\infty}n^\alpha(a_n-1)=\lambda\delta$, we thus have 
\[\lim_{n\to\infty}n^\alpha b_n^\alpha =\lim_{n\to\infty}n^\alpha(1-1/a_n^2)\delta^{-1}=2\lambda,\] 
which concludes the proof.

\end{proof}

In order to investigate the limit of $\Lambda_t^n=\frac{1}{n^{2\alpha}} \int^{nt}_0\lambda_{s}^{n}ds$, we need to study $n^{1-\alpha} \Psi^n(nt)$. To do so, we will look at its Laplace transform inspired by \cite[Lemma 4.3]{MR3563196}.
\begin{lemma}\label{KLP}
Assume \eqref{H1} and \eqref{H2}, recall $\{b_n\}_{n\ge 1}$ introduced in \eqref{bT} and $\Psi^n$ defined in \eqref{PSI2}.  Define 
\begin{align}\label{kn}
K^n(t):=n^{1-\alpha} \Psi^n(nt), \ t\in[0,1].
\end{align}
Then the sequence of measures with density $K^n(t)$ converges weakly to the measure with density $f^{\alpha,\lambda}(t)$ defined in \eqref{MIF}. In particular, for $t\in[0,1]$, 
\[F^n(t)=\int_0^t K^n(s)ds\]
converges uniformly to 
\[F^{\alpha,\lambda}(t)=\frac{1}{\lambda\delta}\int_0^t f^{\alpha,\lambda}(s) ds.\]
\end{lemma}

\begin{proof}
Recalling $ \T \varphi^n(t)=e^{-b_n t}\varphi^n(t)$  and $\T \Psi^n(t)=\sum_{k\ge 1}(\T\varphi^n)^{*k}(t)$,  the Laplace transform of $\T \Psi^n(t)$ is given by 
\[\mL \T\Psi^n(z)=\int_0^\infty e^{-zt}\T \Psi^n(t) dt=\sum_{k\ge 1}\Big(\mL\T\varphi^n(z)\Big)^k, \  z\ge 0.\]
Thus, we need to compute 
\[\mL\T\varphi^n(z)=\int_0^\infty e^{-zt}e^{-b_n t}\varphi^n(t) dt=a_n\mL\varphi(z+b_n).\]
Whence, we have
\[\mL \T\Psi^n(z)=\frac{a_n\mL\varphi(z+b_n)}{1-a_n\mL\varphi(z+b_n)}=\frac{a_n\mL\varphi(z+b_n)}{1-a_n+a_n \Big(1-\mL\varphi(z+b_n) \Big)}.\]
Since $\Psi^n(t)= e^{b_n t}   \T\Psi^n(t)$, we have 
\[K^n(t)=n^{1-\alpha}e^{b_nn t}   \T\Psi^n(nt).\]
Using change of variable ($nt=s$) and the above equality for $\mL \T\Psi^n(z)$, we have for $z\ge nb_n$,
\begin{align*}
\mL K^n(z)&=\int_0^\infty e^{-zt}e^{nb_nt}n^{1-\alpha}\T \Psi^n(nt) dt\\
&=n^{-\alpha}\mL\T\Psi^n(z/n-b_n)\\
&=n^{-\alpha}\frac{a_n\mL\varphi(z/n)}{1-a_n+a_n \Big(1-\mL\varphi(z/n) \Big)}\\
&=\frac{a_n\mL\varphi(z/n)}{-n^{\alpha}(a_n-1)+a_n n^{\alpha}\Big(1-\mL\varphi(z/n)\Big)}.
\end{align*}
Following from 
$$\mL\varphi(z)=1-\delta z^\alpha+o(z^\alpha),\   z\to 0^+,$$ 
with $\delta=K\frac{\Gamma(1-\alpha)}{\alpha}$, proved in  Lemma \ref{BT}-\eqref{LAP},  we deduce that 
\[\lim_{n\to\infty}a_n\mL\varphi(z/n)=1, \]
and that
\[\lim_{n\to\infty}a_nn^\alpha\Big(1-\mL\varphi(z/n)\Big)=\delta z^\alpha,\]
because $\lim\limits_{n\to\infty}a_n=1$ and that 
$$
\lim_{n\to\infty} a_nn^{\alpha}o\big((z/n)^\alpha\big)=0,
$$
which can be easily checked.
Also recall $\lim\limits_{n\to\infty}n^\alpha(a_n-1)=\lambda\delta$ assumed in \eqref{H2}. Consequently, we have 
\begin{align*}
\lim_{n\to\infty}\mL K^n(z)
=\frac{1}{(-\lambda+z^\alpha)\delta}=\mL \Big( \frac{1}{\lambda\delta} f^{\alpha,\lambda} \Big)(z),
\end{align*}
which implies that the measures with density $K^n(t)$ converges weakly to the measure with density $ \frac{1}{\lambda\delta}f^{\alpha,\lambda}(t)$. On the interval $[0,1]$, the measures with density $K^n$ and the measure with density $ \frac{1}{\lambda\delta}f^{\alpha,\lambda}$ are finite, which, along with the continuity of  $F^{\alpha,\lambda}$ gives the uniform convergence of $F^n$ to $F^{\alpha,\lambda}$.

\end{proof}

\vip

Now, we are going to give some basic estimation for the expectation of the Hawkes process and the rescaled intensity process,  which is crucial to prove
the tightness of the renormalized Hawkes process $X^n$ and its related integrated intensity $\Lambda^n$ introduced in \ref{notation}.
\begin{lemma}\label{estiH} Assume \eqref{H1} and \eqref{H2}. Consider the solution $(Z_t^n)_{t\ge0}$ to \eqref{sssy} and $\lambda^{n}$ defined in \eqref{sssy}. We have 
 \begin{enumerate} 
 \item There is $C>0$ independent of  $n$ and $t$ such that
 $$\sup\limits_{t\in[0,1]}\E[Z_{nt}^{n}]
 \le Cn^{2\alpha}.$$
\item  
There is $C>0$ independent of  $n$ and $t$ such that 
\[\sup_{n\ge1}\sup_{t\in[0,1]}\E\Big[\frac{\lambda_{nt}^n}{n^{2\alpha-1}}\Big]\le C.\]
\end{enumerate}
\end{lemma}

\begin{proof} 
It follows from \cite[Lemma 22]{MR3449317} that the Hawkes process $(Z_t^n)_{t\ge0}$ can be rewritten as 
\[Z_t^n=M_t^n+\mu_n t+\int_0^t  \varphi^n(t-s)Z_s^nds.\]
According to  \cite[Lemma 3]{MR3054533}, see also \cite[Lemma 8]{MR3499526}, we then have  for $t\ge0$,
\begin{align}\label{expZ}
\E[Z_t^n]=\mu_n t+\mu_n\int_0^t s\Psi^n(t-s)\ ds.
\end{align}
Replacing $t$ by $nt$ in \eqref{expZ}, we obtain for any $t\in[0,1]$
\begin{align*}
    \E[Z_{nt}^{n}]&=\mu_n nt+\mu_n \int_{0}^{nt}s\Psi^n(nt-s)ds\\
   & \leq\mu_n nt+\mu_n nt\int_{0}^{nt}\Psi^n(nt-s)ds\\
   &\leq\mu_n nt+\mu_n nt\int_{0}^{nt} e^{b_n (nt-s)}\T \Psi^n(nt-s)ds\\
   &\leq\mu_n nt+\mu_n nt e^{b_n nt}\|\T \Psi^n\|_1\le\mu_n n e^{nb_n}\frac{a_n}{a_n-1},
    \end{align*}
    which completes the proof of (1) under the assumption $\eqref{H1}$, $\eqref{H2}$ together with Lemma \ref{BT}. 
\vip
For (2), according to \cite[Proposition 2.1]{MR3313750}, we know that $\lambda^n$ has the following representation, for $t\ge0$,
\[\lambda_t^n=\mu_n+\mu_n\int_0^t\Psi^n(t-s)ds+\int_0^t\Psi^n(t-s)dM_s^n,\] where $M^n$ is defined in \eqref{mart}. For any $t\in[0,1]$,
\begin{align*}
    \E[\lambda_{nt}^{n}]&=\mu_n+\mu_n \int_{0}^{nt}\Psi^n (nt-s)ds\\
    &=\mu_n+\mu_n \int_{0}^{nt}e^{b_n(nt-s)}\T \Psi^n(nt-s)ds\\
    &\le \mu_n e^{b_n nt}(1+\|\T \Psi^n\|_1)\le  \mu_n e^{b_n n}\frac{a_n}{a_n-1}.
 \end{align*}
 Since \eqref{H1}, \eqref{H2} and Lemma \eqref{BT},  $\frac{e^{b_nn}\mu_na_n}{n^{2\alpha-1}(a_n-1)}\to\frac{\mu^*}{\lambda}e^{(2\lambda)^{1/\alpha}}$ as $n\to\infty$,
 which gives the proof of (2).
\end{proof}

Next, we prove that  the limit of $X^n$  is equal to that of $\Lambda^n$ which means that we can work with  $\Lambda^n$ instead of $X^n$.
\vip
\begin{lemma}\label{CIP} Recall the renormalized Hawkes process $X_t^n=Z^n_{nt}/n^{2\alpha}$ and the integrated intensity  $\Lambda_t^n=\frac{1}{n^{2\alpha}} \int^{nt}_0\lambda_{s}^{n}ds$ for $t\in[0,1]$ introduced in \ref{notation}. Under the assumption \eqref{H1} and \eqref{H2}, the sequence of martingales $X^n-\Lambda^n$ converges to zero in probability uniformly in time.
\end{lemma}
\begin{proof}
In order to prove  $X^n-\Lambda^n$ converging to zero in probability uniformly in time, it suffices to prove that $X^n-\Lambda^n$ converges to zero in $L^2$ uniformly in time. 
Recalling  $M_t^n=Z_{t}^{n}-\int_0^t \lambda^n_s ds$, we find the difference 
\[X^n_t-\Lambda^n_t=\frac{M_{nt}^n}{n^{2\alpha}}.\]
Applying the Doob's martingale inequality to the martingale $M^n$ and using $[M^n,M^n]_t=Z_t^n$ for any $t\ge0$, we have 
\begin{align*}
\E\Big[\sup_{t\in[0,1]}(X^n_t-\Lambda^n_t)^2\Big] 
\le \frac{4}{n^{4\alpha}}\E\big[\big(M_n^n\big)^2\big]
=\frac{4}{n^{4\alpha}}\E\big[[M^n,M^n]_n\big]
=\frac{4}{n^{4\alpha}}\E\big[Z^n_n\big].
\end{align*}
By Lemma \ref{BT},  \ref{estiH} and \eqref{H2}, we conclude that 
\begin{align*}
\E\Big[\sup_{t\in[0,1]}(X^n_t-\Lambda^n_t)^2\Big] 
\le \frac{C}{n^{2\alpha}}\to 0,
\end{align*}
as $n\to\infty$. Hence, we complete the proof.
\end{proof}

\begin{remark} If $X^n\stackrel{d}\longrightarrow X$, then Lemma \ref{CIP} implies that $\Lambda^n\stackrel{d}\longrightarrow X$.
\end{remark}

We are now prepared for proving the weak convergence of $(\overline{M}^n,X^n)$ introduced in section \ref{notation} in the sense of  the Skorohod  topology.
\begin{prop}\label{limpoi}
Under the assumption \eqref{H1} and \eqref{H2},  the sequence $(\overline{M}^n,X^n)$ is tight. Moreover, if a process $(M^*,X)$ is a limit  of $(\overline{M}^n,X^n)$, then $M^*$ is a continuous martingale with $[M^*,M^*]=<M^*,M^*>=X$.
\end{prop}
\begin{proof} We will divide the proof into two steps.
\vip
{\it{Step 1.}} The first step is to prove the  $\mathbb{C}$-tightness of $X^n$ and $\Lambda^n$ (i.e. they are tight and their limits are continuous) in $D([0,1],\R_+)$ equipped with the Skorohod topology.  Recalling $X_t^n=Z^n_{nt}/n^{2\alpha}$, we then deduce from Lemma \ref{estiH} that there exists  $C>0$ independent of $n$ and $t$ such that 
 \begin{align*}
        \E[\Lambda_{1}^{n}]=\E[X_{1}^{n}]=\frac{1}{n^{2\alpha}}\E[Z^n_n]\le C.
    \end{align*}
Whence,  the processes $(X_t^n)_{t\in[0,1]}$ and $(\Lambda_t^n)_{t\in[0,1]}$ are tight thanks to the fact that both processes are increasing. Following from \cite[Proposition VI-3.26]{MR1943877},  it's thus clear  that $(\Lambda_t^n)_{t\in[0,1]}$ and $(X_t^n)_{t\in[0,1]}$ are $\mathbb{C}$-tight  because $\Lambda_t^n$ is continuous for $t\in[0,1]$ (the jump size is zero) and
\[\clim\sup_{t\in[0,1]}|\Delta X_t^n|=\clim\frac{1}{n^{2\alpha}}=0,\]
i.e., the maximum jump size of $X^n$ converges to zero as $n$ goes to infinity.
\vip
{\it{Step 2.}} 
Noting that $<\overline{M}^n, \overline{M}^n>_t=\Lambda^n_t$ for $t\in[0,1]$ together with that the  $\mathbb{C}$-tightness of $\Lambda^n$ is proved in {\it{Step 1}}, the tightness of $(\overline{M}^n_t)_{t\in[0,1]}$ is straightforward followed from \cite[Theorem VI-4.13]{MR1943877}, which furthermore implies the tightness of $(\overline{M}^n, X^n)$. By Prokhorov's theorem,  there exists some increasing subsequence still denoted by $\{n\}_{n\ge 1}$ and a limit point $(M^*,X)$ such that   $(\overline{M}^{n},X^{n})\stackrel{d}\longrightarrow(M^*,X)$. Since each $\overline{M}^{n}$ is a martingale, it's of course a local martingale, and its maximum jump size $\sup\limits_{t\in[0,1]}|\Delta \overline{M}^{n}_t|=\frac{1}{n^\alpha}\to0$ as $n\to\infty$, we thus deduce from \cite[Proposition VI-6.26]{MR1943877} that $[\overline{M}^{n},\overline{M}^{n}]\stackrel{d}\longrightarrow[M^*,M^*]$, which, together with  $[\overline{M}^{n},\overline{M}^{n}]=X^{n}$ and  $X^{n}\stackrel{d}\longrightarrow X$, we get $[M^*,M^*]=X$. Moreover, from \cite[Proposition VI-3.26]{MR1943877},  we have  $\overline{M}^{n}$ is $\mathbb{C}$-tight since the maximum jump size of $\overline{M}^{n}$ goes to zero, which immediately implies  $M^*$ is continuous.  
\vip
Finally, it remians to prove that $M^*$ is a martingale. From \cite[Corollary IX.1.19]{MR1943877}, we know that $M^*$ is a local martingale. In order to prove that a local martingale $M^*$ is a martingale,  it suffices to check 
\[\E\big[\sup_{s\in[0,t]}|M^*_s|\big]<\infty,\, \hbox{fo\  every\  $t$.}\] Indeed, by the Burkholder–Davis–Gundy inequality and the Cauchy–Schwarz inequality, there is $C$ such that for any $t\in[0,1]$,
\[\E\big[\sup_{s\in[0,t]}|M^*_s|\big]\le C\E\big[[M^*,M^*]_t^{1/2}\big]\le C\Big(\E\big[[M^*,M^*]_t\big]\Big)^{1/2}=C\Big(\E\big[X_t\big]\Big)^{1/2}\le C.\]
The fact  $\E[X]$ being bounded is due to the Fatou's Lemma, i.e.
\[\E[X_t]\le \liminf_{n\to\infty}\E[X^{n}_t]\le  \liminf_{n\to\infty}\E[X^{n}_1]\le C.\]
We complete the proof.
\end{proof}

\section{Proof of the main results for the univariate case}

By the Skorohod representation theorem in \cite[p.9,Theorem 2.7]{MR1011252}, without loss of generality, we assume that the limit in above proposition holds almost surely, and we still use $\{n\}$ to represent its subsequence $\{n_i\}$ for convenience.  Since $(M^*,X)$ are continuous, we have 
\begin{align}\label{as}
\lim_{n\to\infty}\sup_{t\in[0,1]}|\overline{M}_t^n-M_t^*|=0,\  a.s.
\end{align}
\begin{proof}[Proof of Theorem \ref{thm1}]
According to Lemma \ref{CIP}, in order to get the limit of $X^n$, it's equivalent to look at the limit of $\Lambda^n$. We will proceed the proof  in steps. 

\vip
{\it Step 1}. First, we will rewrite $\Lambda_t^n=\frac{1}{n^{2\alpha}} \int^{nt}_0\lambda_{s}^{n}ds$. Recalling  the intensity $\lambda_{t}^{n}=\mu_n+\int_{0}^{t}\varphi^n(t-s)dZ_{s}^{n}$  and  using twice the  Fubini theorem, we have for $t\in[0,1]$,
\begin{align*}
\int^{t}_0\lambda_{s}^{n}ds=\mu_nt+\int_{0}^{t}\varphi^n(t-s)M_{s}^{n}ds+\int_{0}^{t}\int_{0}^{s}\varphi^n(t-s)\lambda_{v}^{n}dvds.
\end{align*}
Then, according to \cite[Lemma 3]{MR3054533}, see also \cite[Lemma 8]{MR3499526},  and  $\Psi^n(t)$ defined in \eqref{PSI2}, we have 
\begin{align*}
\int^{t}_0\lambda_{s}^{n}ds&=\mu_nt+\int_{0}^{t}\varphi^n(t-s)M_{s}^{n}ds+\sum_{n=1}^\infty\int_0^t (\varphi^n)^{*n}(t-s) \Big( \mu_ns+\int_{0}^{s}\varphi^n(s-r)M_{r}^{n}dr \Big)ds\\
&=\mu_nt+\int_{0}^{t}\varphi^n(t-s)M_{s}^{n}ds+\int_0^t \Psi^n(t-s)\mu_ns ds +\int_0^t \Psi^n(t-s)\int_{0}^{s}\varphi^n(s-r)M_{r}^{n}dr ds.
\end{align*}
Following from the fact that
\begin{align*}
\int_0^t \Psi^n(t-s) \int_0^s \varphi^n(s-r)M_r^n drds=\int_0^t \Psi^n(t-r)M_r^n dr-\int_0^t \varphi^n(t-r)M_r^n dr,
\end{align*}
proved in \cite{MR3563196}, we have
\begin{align*}
\int^{t}_0\lambda_{s}^{n}ds=\mu_nt+\int_0^t \Psi^n(t-s)s\mu_n ds+\int_0^t \Psi^n(t-s)M_s^n ds.
\end{align*}
Now, a change of variable leads to
 \begin{align*}
\Lambda_t^n=\frac{\mu_nt}{n^{2\alpha-1}}+\frac{\mu_n}{n^{2\alpha-2}}\int_0^{t} \Psi^n(n(t-s))s ds+\frac{1}{n^{2\alpha-1}}\int_0^{t} \Psi^n(n(t-s))M_{ns}^n ds.
\end{align*} 

\vip

{\it Step 2}. We are going to compute the limit  of $\Lambda^n$ term by term. First of all, it's obvious that  $\lim\limits_{n\to\infty}\frac{\mu_nt}{n^{2\alpha-1}}=0$ for any $t\in[0,1]$ since $\lim\limits_{n\to\infty} n^{1-\alpha}\mu_n=\mu^*\delta$ in \eqref{H2}. For the second term, we denote $I_t^n:=\frac{\mu_n}{n^{2\alpha-2}}\int_0^{t} \Psi^n(n(t-s))s ds$.  Then, 
\[\lim_{n\to\infty} \sup_{t\in[0,1]}\left|I_t^n-\frac{\mu^*}{\lambda}\int_0^t f^{\alpha,\lambda}(t-s)sds\right|=0.\]
 Recalling $K^n$ in Lemma  \ref{KLP}-\eqref{kn}, we have 
\begin{align*}
I_t^n=\mu_nn^{1-\alpha}\int_0^t K^n(t-s)sds.
\end{align*}
Integrating by parts and recalling $F^n, F^{\alpha,\lambda}$ defined  in Lemma \ref{KLP}, we have 
\begin{align*}
&I_t^n =\mu_nn^{1-\alpha}\int_0^t F^n(t-s)ds.\\
&\frac{\mu^*}{\lambda}\int_0^t f^{\alpha,\lambda}(t-s)sds =\mu^*\delta\int_0^t F^{\alpha,\lambda}(t-s)ds.
\end{align*}
Using  the triangle inequality together with  \eqref{H2} and Lemma \ref{KLP}, we have for any $t\in[0,1]$,
\[\left|I_t^n-\frac{\mu^*}{\lambda}\int_0^t f^{\alpha,\lambda}(t-s)sds\right|\to 0,\]
as $n$ goes to infinity.
\vip

{\it Step 3}. Finally,  let's study the convergence of  $J_t^n:=\frac{1}{n^{\alpha-1}}\int_0^{t} \Psi^n(n(t-s))\overline{M}_{s}^n ds$ with $\overline{M}^n_t=M^n_{nt}/n^{\alpha}$. Set
\[V_t:=\frac{1}{\lambda\delta}\int_0^tf^{\alpha,\lambda}(t-s)M^*_s ds, \  t\in[0,1].\]
 Recalling $K^n$ defined in \eqref{kn}, $J_t^n$ can be written as 
 \begin{align*}
J_t^n&=\int_0^t K^n(t-s)\overline{M}^n_s ds.
\end{align*}
Then, $|J_t^n-V_t|\le|\Delta_1^n(t)|+| \Delta_2^n(t)|$ with 
\begin{align*}
\Delta_1^n(t)&:=\int_0^t \Big(K^n(t-s)-\frac{1}{\lambda\delta}f^{\alpha,\lambda}(t-s)\Big)\overline{M}_s^n ds,\\
\Delta_2^n(t)&:=\frac{1}{\lambda\delta} \int_0^t f^{\alpha,\lambda}(t-s) (\overline{M}_s^n-M^*_s) ds.
\end{align*}
According to \eqref{as}, we easily conclude that 
\[
\sup_{t\in[0,1]}|\Delta_2^n(t)|\le \sup_{s\in[0,1]}|\overline{M}_s^n-M^*_s|   \frac{1}{\lambda\delta} \int_0^t f^{\alpha,\lambda}(t-s)ds\to 0, \ a.s.
\]
Furthermore, for $\Delta_1^n(t)$, using $<\overline{M}^n, \overline{M}^n>_t=\Lambda^n_t$ for $t\in[0,1]$ and the Stieltjes integration by parts, we get 
\begin{align*}
\E[|\Delta_1^n(t)|^2]&=\E\left[\left(\int_0^t \Big(K^n(t-s)-\frac{1}{\lambda\delta}f^{\alpha,\lambda}(t-s)\Big)\overline{M}^n_s ds\right)^2\right]
\\
&=\E\left[\left(\int_0^t \Big(F^n(t-s)-F^{\alpha,\lambda}(t-s)\Big)d\overline{M}^n_s \right)^2\right]\\
&= \E\left[ \int_0^t \Big(F^n(t-s) - F^{\alpha,\lambda}(t-s)\Big)^2d\Lambda^n_s \right]\\
&=\left[ \int_0^t \Big(F^n(t-s) - F^{\alpha,\lambda}(t-s)\Big)^2 \frac{1}{n^{2\alpha-1}} \E[\lambda^n_{ns}] ds\right] \  \Big(\hbox{Recalling}\  \Lambda_t^n=\frac{1}{n^{2\alpha}} \int^{nt}_0\lambda_{s}^{n}ds \Big)\\
&\le C\int_0^t \Big(F^n(t-s)-F^{\alpha,\lambda}(t-s)\Big)^2ds,
\end{align*}
where $C$ is a positive constant independent of $n$ and $t$. The above last inequality holds due to Lemma \ref{estiH}-(2). Then following from the dominated convergence theorem and  the uniform convergence proved in Lemma \ref{KLP} that
\begin{align*}
\lim_{n\to\infty}\E[|\Delta_1^n(t)|^2]=0.
\end{align*}
 Whence, we conclude that  $J_t^n$ converges in law  to $V_t$ for $t\in[0,1]$, which,  together with previous steps, completes the proof.
 
\end{proof}
Next, we give the 
\begin{proof}[Proof of theorem \ref{thm2}] According to Remark \ref{RK}, we know that $Y$ is continuous and $X$ can be written as
$$X_t=\int_0^t Y_sds $$ for $t\in[0,1]$. Since $<M^*,M^*>=X$ and $M^*$ is continuous, we thus conclude from the martingale representation theorem, see e.g. \cite[Theorem 7.1]{MR1011252},  that
\[M^*_t=\int_0^t \sqrt{Y_s}dB_s,\] where $B$ is a Brownian motion. By Theorem \ref{thm1}, we know that $X$ satisfies 
$$X_t=\frac{\mu^*}{\lambda}\int_0^t f^{\alpha,\lambda}(t-s)s ds+\frac{1}{\lambda\delta}\int_0^t f^{\alpha,\lambda}(t-s)M^*_sds.
$$ 
With $M^*$ replaced by$ \int_0^\cdot \sqrt{Y_s}dB_s$, we have
\[X_t=\frac{\mu^*}{\lambda}\int_0^t f^{\alpha,\lambda}(t-s)s ds+\frac{1}{\lambda\delta}\int_0^t f^{\alpha,\lambda}(t-s)\int_0^s \sqrt{Y_r}dB_rds.\]
 Then, integration by parts and a change of variable give us 
\[\frac{1}{\lambda\delta}\int_0^t f^{\alpha,\lambda}(t-s)s ds=\int_0^t F^{\alpha,\lambda}(t-s)ds=\int_0^t F^{\alpha,\lambda}(s)ds.\]
Using twice the Fubini theorem , we  have 
\[\int_0^t f^{\alpha,\lambda}(t-s)\int_0^s \sqrt{Y_r}dB_rds=\int_0^t \int_0^s f^{\alpha,\lambda}(s-r) \sqrt{Y_r}dB_rds.\]
Thus, $X$ satisfies 
\[X_t=\mu^*\delta \int_0^t F^{\alpha,\lambda}(s)ds+\frac{1}{\lambda\delta}\int_0^t \int_0^s f^{\alpha,\lambda}(s-r) \sqrt{Y_r}dB_rds.\]
Since $X$ is differentiable and the integrands of  the preceding equality are continuous, Theorem \ref{thm2}  holds by differentiating both sides of the above equality with respect to $t$. The uniqueness in law of the equation \eqref{Vpro} follows directly from \cite[Theorem 6.1]{MR4019885}.
\end{proof}

\section*{Acknowledgements}
\quad~~
Liping Xu is supported by National Natural Science Foundation of China (12101028). 
\bibliographystyle{abbrv} 
\bibliography{heavy}

\end{CJK}
\end{document}